\documentclass[letterpaper, 10pt, conference]{ieeeconf}
\IEEEoverridecommandlockouts
\usepackage{amsmath,amssymb,amsfonts,amsthm}
\usepackage{graphicx}
\usepackage{textcomp}
\usepackage{xcolor}
\usepackage{cite}
\usepackage{bm}
\usepackage{mathtools}
\newtheorem{theorem}{Theorem}
\newtheorem{corollary}{Corollary}
\newtheorem{lemma}{Lemma}

\newtheorem{remark}{Remark}
\newtheorem{definition}{Definition}
\newcommand{\R}{\mathbb{R}}
\newcommand{\C}{\mathbb{C}}
\newcommand{\N}{\mathcal{N}}
\newcommand{\im}{\mathrm{Im}}

\begin{document}
\title{Bessel Function Analysis of Nesterov's ODE in $N$-Player Quadratic Games}
\author{Jay Paek$^{1}$
\thanks{$^{1}$J. Paek is with University of California, San Diego,
La Jolla, CA 92093 USA. {\tt jpaek@ucsd.edu}}}
\maketitle

\begin{abstract}
We analyze Nesterov's accelerated gradient dynamics for pseudo-gradient equilibrium seeking in \(N\)-player quadratic games. When each player's cost is convex in its own action, the pseudo-gradient stationary points considered here are Nash equilibria. While the continuous-time NAGD dynamics---the Su--Boyd--Candès ODE---are well understood for convex optimization, their behavior with non-symmetric pseudo-gradient matrices arising in games is substantially different.

Using a Bessel-function modal analysis, we show that a stationary point is generically unstable whenever any eigenvalue of the pseudo-gradient matrix \(G\) lies outside \(\mathbb R_{\ge 0}\). Conversely, if every eigenvalue lies in \(\mathbb R_{\ge 0}\) and \(G\) is diagonalizable, all trajectories are bounded and converge to \(\mathcal N(G)\), with the components transverse to \(\mathcal N(G)\) decaying at rate \(O(t^{-3/2})\). The diagonalizability hypothesis is essential: defective Jordan blocks with nonnegative eigenvalues can change the polynomial behavior and may cause divergence. In particular, complex eigenvalues with positive real parts, which stabilize first-order gradient dynamics, induce exponential instability in NAGD. This shows that the momentum mechanism responsible for acceleration in optimization can be detrimental for equilibrium seeking in non-potential games.
\end{abstract}

\section{Introduction}
Nesterov's accelerated gradient descent \cite{nesterov1983} achieves the optimal $O(1/k^2)$ convergence rate among first-order methods for smooth convex optimization. Su, Boyd, and Cand\`{e}s \cite{su2016} established that its continuous-time limit is
\begin{equation}
\ddot{X}(t) + \frac{r}{t}\dot{X}(t) + \nabla f(X(t)) = 0, \label{eq:SBC_ODE}
\end{equation}
where $r \geq 3$ is a damping parameter. For $r = 3$ and quadratic objectives $f(x) = \frac{1}{2}x^\top A x$ with $A \succ 0$, the solution can be expressed via Bessel functions, and convergence to the minimizer is guaranteed \cite{su2016}.

This paper studies~\eqref{eq:SBC_ODE} when the symmetric gradient $\nabla f$ is replaced by a non-symmetric pseudo-gradient arising in $N$-player games. Each player minimizes their own cost, leading to a pseudo-gradient matrix $G$ that is generically non-symmetric except in potential games \cite{monderer1996}. The instability of Nesterov-type dynamics in non-potential games has been studied from several angles---via multi-time-scale analysis \cite{poveda2025} over asymmetric graphs \cite{jakovetic2020}, and via hybrid restarting mechanisms \cite{poveda2023}. Our work complements these contributions by developing a \emph{Bessel function modal analysis} that yields explicit closed-form solutions, sharp spectral thresholds, and precise growth/decay rates for the quadratic case.

\subsection{Contributions}
We decompose the NAGD game dynamics into scalar modal equations and solve them explicitly using Bessel functions of complex argument, extending the quadratic analysis of \cite{su2016} from symmetric Hessians to possibly non-symmetric pseudo-gradient matrices. This yields two sharp statements. 
First, if any eigenvalue of $G$ lies outside $\mathbb{R}_\{\ge 0\}$, then the equilibrium is unstable; indeed, outside a proper linear exceptional set of initial conditions, some modal component grows with an exponential factor. Second, if all eigenvalues lie in $\mathbb{R}_\{\ge 0\}$ and $G$ is diagonalizable, then all trajectories are bounded and converge to $\mathcal{N} (G)$, with the nonzero-eigenvalue components decaying at rate $O(t^\{\-3\/2\})$. The convergence direction is therefore a diagonalizable-matrix result, not a purely eigenvalue-only statement for arbitrary matrices.

\subsection{Related Work}
The ODE perspective on Nesterov's method was pioneered by Su, Boyd, and Cand\`{e}s \cite{su2016}, and extended to Hessian-driven damping \cite{attouch2018}, high-resolution ODEs \cite{shi2022}, and Lyapunov-based analysis \cite{wilson2021}; these works consider optimization with symmetric Hessians. Ochoa et al.~\cite{poveda2023} studied Nesterov's ODE for games using Lyapunov and multi-time-scale techniques, and proposed hybrid restarting mechanisms; instability for general non-conservative mappings was further analyzed in \cite{poveda2025}. 

While \cite{poveda2025} establishes instability of Nesterov's ODE under non-conservative mappings using averaging and Floquet-theoretic ideas, our Bessel-function approach gives complementary information in the quadratic case: explicit closed-form modal solutions, precise exponential growth rates such as \(|\operatorname{Im}\sqrt{\lambda}|\) for complex eigenvalues, and an eigenvalue-level instability criterion. The matching convergence statement holds under the additional hypothesis that \(G\) is diagonalizable. Thus the present result is a complete spectral characterization only within the diagonalizable class; defective Jordan blocks with eigenvalues in \(\mathbb R_{\ge 0}\) require separate polynomial analysis and can invalidate an eigenvalue-only convergence claim. We do not address the hybrid restarting mechanisms proposed in \cite{poveda2023}.

Jakoveti\'{c} et al.~\cite{jakovetic2020} analyzed Nesterov-type acceleration over directed graphs. Stability of first-order gradient dynamics in games has been studied in \cite{rosen1965,hofbauer2009,mertikopoulos2019}, with spectral characterizations by Chasnov et al.~\cite{chasnov2020, chasnov2020b} and local convergence results by Mazumdar et al.~\cite{mazumdar2020}. Recent work has applied Nesterov-type acceleration to distributed Nash equilibrium seeking \cite{liu2024,wang2025}, typically assuming strong monotonicity.

\subsection{Organization}
Section~\ref{sec:prelim} establishes the problem setting. Section~\ref{sec:projection} develops the Bessel function analysis. Section~\ref{sec:main} presents the main spectral characterization. Section~\ref{sec:simulations} provides numerical illustrations, and Section~\ref{sec:discussion} concludes.

\section{Preliminaries}
\label{sec:prelim}

\subsection{Notation}
Let $\R_{\geq 0}$ denote non-negative reals. For $A \in \R^{n \times n}$, the null space is $\N(A) = \{x : Ax = 0\}$ and the condition number is $\kappa(A) = \|A\|\|A^{-1}\|$. For $z \in \C$, we write $\Re(z)$ and $\im(z)$ for real and imaginary parts. The principal square root $\sqrt{\lambda}$ is chosen with $\Re(\sqrt{\lambda}) \geq 0$.

\subsection{Quadratic Games and the Pseudo-Gradient}
Consider an \(N\)-player scalar-action quadratic game in which player
\(i\) controls \(x_i\in\mathbb R\) and has cost
\[
J_i(x)=x^\top Q_i x+d_i^\top x,
\qquad
x=(x_1,\ldots,x_N)^\top,
\]
where \(Q_i\in\mathbb R^{N\times N}\) is symmetric and
\(d_i\in\mathbb R^N\). This formulation allows each player's cost to
depend on the full joint action profile.

\begin{definition}[Pseudo-gradient]
The pseudo-gradient \(F:\mathbb R^N\to\mathbb R^N\) has components
\[
F_i(x)=\frac{\partial J_i}{\partial x_i}(x).
\]
\end{definition}

For the quadratic costs above,
\[
F_i(x)=2\sum_{j=1}^N (Q_i)_{ij}x_j+(d_i)_i,
\]
so \(F(x)=Gx+b\), where
\[
G_{ij}=2(Q_i)_{ij},
\qquad
b_i=(d_i)_i .
\]
We call any point \(x^\star\) satisfying \(F(x^\star)=0\) a
pseudo-gradient stationary point, or pseudo-gradient equilibrium.

If each player is convex in its own action, equivalently
\[
\frac{\partial^2 J_i}{\partial x_i^2}=2(Q_i)_{ii}=G_{ii}\ge 0
\qquad \text{for all } i,
\]
then any pseudo-gradient equilibrium is a Nash equilibrium for the unconstrained scalar-action game. Without this own-action convexity condition, \(F(x^\star)=0\) is only a stationary point and need not be a Nash equilibrium.

After shifting coordinates by a pseudo-gradient equilibrium
\(x^\star\), we assume without loss of generality that \(b=0\) and
\(x^\star=0\). The matrix \(G\) is symmetric if and only if the pseudo-gradient is the gradient of a quadratic potential \cite{monderer1996}.

\subsection{NAGD Game Dynamics}
Applying~\eqref{eq:SBC_ODE} with $r = 3$ and the pseudo-gradient yields
\begin{equation}
\ddot{x}(t) + \frac{3}{t}\dot{x}(t) + Gx(t) = 0, \quad t \in [t_0, \infty), \; t_0 > 0. \label{eq:NAGD_game}
\end{equation}
We work on $[t_0, \infty)$ to avoid the singularity at $t = 0$, and focus on $r = 3$, which is optimal for $O(1/t^2)$ convergence in optimization \cite{su2016}. For comparison, the first-order gradient dynamics \(\dot x=-Gx\) are
asymptotically stable if and only if all eigenvalues of \(G\) have positive real parts. For the NAGD dynamics below, any eigenvalue outside \(\mathbb R_{\ge 0}\) already produces generic exponential instability.
Conversely, when \(G\) is diagonalizable, the condition \(\sigma(G)\subset \mathbb R_{\ge 0}\) is sufficient for convergence. Without diagonalizability, eigenvalues alone do not determine convergence.

\section{Modal Decomposition and Bessel Function Solutions}
\label{sec:projection}

Our approach projects the dynamics onto eigenspaces of $G$, yielding scalar Bessel-type equations. We use \emph{left} eigenvectors because $G$ is generically non-symmetric.

\subsection{Eigenvector Projection}
\begin{lemma}[Modal projection]
\label{lem:modal}
Let $w \in \C^N \setminus \{0\}$ satisfy $w^*G = \lambda w^*$. Then $y(t) = w^* x(t)$ satisfies
\begin{equation}
\ddot{y}(t) + \frac{3}{t}\dot{y}(t) + \lambda y(t) = 0. \label{eq:modal}
\end{equation}
\end{lemma}
\begin{proof}
Apply $w^*$ to~\eqref{eq:NAGD_game} and use $w^*G = \lambda w^*$.
\end{proof}

\subsection{Bessel Function Solutions}
\begin{lemma}[General solution]
\label{lem:bessel}
For $\lambda \neq 0$, the general solution of~\eqref{eq:modal} is
\begin{equation}
y(t) = \frac{1}{t}\left[c_1 J_1(\sqrt{\lambda}\, t) + c_2 Y_1(\sqrt{\lambda}\, t)\right], \label{eq:bessel_sol}
\end{equation}
where $J_1, Y_1$ are Bessel functions of the first and second kind of order $1$, and $c_1, c_2 \in \C$ are determined by initial conditions. For $\lambda = 0$: $y(t) = c_1 + c_2 t^{-2}$.
\end{lemma}
\begin{proof}
The substitution $y(t) = t^{-1}u(t)$ transforms~\eqref{eq:modal} into Bessel's equation of order $1$ \cite{olver2010}. The $\lambda = 0$ case follows by direct integration.
\end{proof}

\subsection{Asymptotic Behavior}
\begin{lemma}[Modal asymptotics]
\label{lem:asymptotics}
Let $y(t)$ solve~\eqref{eq:modal} with
$(y(t_0), \dot{y}(t_0)) \in \C^2$.
\begin{enumerate}
\item[(i)] $\lambda > 0$: $y(t) = O(t^{-3/2}) \to 0$ and
$\dot{y}(t) = O(t^{-3/2}) \to 0$.
\item[(ii)] $\lambda < 0$: Writing $\lambda = -\mu^2$, there is a
one-dimensional complex subspace $\mathcal S_\lambda \subset \C^2$ such that, for all initial conditions outside $\mathcal S_\lambda$, $|y(t)| \to \infty$ exponentially at rate $\mu$. For real scalar initial conditions, the exceptional set is a one-dimensional subspace of $\R^2$ and hence has Lebesgue measure zero.
\item[(iii)] $\lambda \in \C \setminus \R$: There is a one-dimensional complex subspace $\mathcal S_\lambda \subset \C^2$ such that, for all initial conditions outside $\mathcal S_\lambda$, $|y(t)| \to \infty$ exponentially at rate $|\im(\sqrt{\lambda})|$. Moreover, $\mathcal S_\lambda \cap \R^2 = \{(0,0)\}$, so every nonzero real scalar initial condition grows exponentially.
\item[(iv)] $\lambda=0$: 
\[
y(t)
=
y(t_0)+\frac{t_0}{2}\dot y(t_0)
\left(1-\frac{t_0^2}{t^2}\right),
\qquad
\dot y(t)=\left(\frac{t_0}{t}\right)^3\dot y(t_0).
\]
In particular, \(y(t)\) converges to a finite limit and \(\dot y(t)\to0\).
\end{enumerate}
\end{lemma}

\begin{proof}
(i) For $\sqrt{\lambda} \in \R_{>0}$, the Bessel asymptotics
$J_1(z),Y_1(z)=O(z^{-1/2})$ as $z\to\infty$ \cite{olver2010} give
$y(t)=O(t^{-3/2})$. Differentiating~\eqref{eq:bessel_sol} and using
$J_1'(z),Y_1'(z)=O(z^{-1/2})$ gives
$\dot y(t)=O(t^{-3/2})$.

(ii) For $\lambda=-\mu^2$, the scalar equation becomes
\[
\ddot y+\frac{3}{t}\dot y-\mu^2 y=0.
\]
Equivalently,
\[
y(t)=\frac{1}{t}\left[A I_1(\mu t)+B K_1(\mu t)\right],
\]
where $I_1$ and $K_1$ are modified Bessel functions. The Wronskian of
$I_1(\mu t)/t$ and $K_1(\mu t)/t$ is nonzero, so the map from
$(y(t_0),\dot y(t_0))$ to $(A,B)$ is a linear isomorphism. Since
\[
I_1(\mu t)\sim \frac{e^{\mu t}}{\sqrt{2\pi\mu t}},
\qquad
K_1(\mu t)\sim \sqrt{\frac{\pi}{2\mu t}}\,e^{-\mu t},
\]
the solution grows exponentially at rate $\mu$ unless $A=0$. Thus the
exceptional set is the one-dimensional complex subspace
$\mathcal S_\lambda=\{(y(t_0),\dot y(t_0)):A=0\}$. When the initial data
are real, the coefficients $A,B$ are real, and $A=0$ defines a real line
in $\R^2$.

(iii) See Appendix~I.

(iv) For $\lambda=0$, let $p(t)=\dot y(t)$. Then
\[
p'(t)+\frac{3}{t}p(t)=0,
\]
so
\[
p(t)=p(t_0)\left(\frac{t_0}{t}\right)^3.
\]
Integrating from $t_0$ to $t$ gives
\[
y(t)
=
y(t_0)+\dot y(t_0)t_0^3\int_{t_0}^t s^{-3}\,ds
=
y(t_0)+\frac{t_0}{2}\dot y(t_0)
\left(1-\frac{t_0^2}{t^2}\right).
\]
\end{proof}

The following corollary summarizes the instability consequence of
Lemma~\ref{lem:asymptotics}; Theorems~\ref{thm:symmetric}--\ref{thm:general}
in Section~\ref{sec:main} will refine this into convergence statements
under the corresponding spectral hypotheses.

\begin{corollary}[Generic instability]
\label{cor:generic}
If $G$ has an eigenvalue $\lambda \notin \R_{\geq 0}$, then for
Lebesgue-almost every initial condition
$(x(t_0),\dot{x}(t_0))\in \R^{2N}$, $\|x(t)\|_2 \to \infty$
exponentially.
\end{corollary}

\begin{proof}
Let $w$ be a left eigenvector for $\lambda$, so $w^*G=\lambda w^*$, and
define the modal-data map
\[
\Phi_w(x_0,v_0)=(w^*x_0,w^*v_0),
\qquad (x_0,v_0)\in\R^N\times\R^N.
\]
The corresponding modal variable $y(t)=w^*x(t)$ satisfies
\eqref{eq:modal}.

First suppose $\lambda<0$. Since $G$ is real, $w$ may be chosen real.
Then $\Phi_w:\R^{2N}\to\R^2$ is onto, and by
Lemma~\ref{lem:asymptotics}(ii) the nongrowing scalar data form a real
line in $\R^2$. Its preimage under $\Phi_w$ is a proper linear subspace
of $\R^{2N}$, hence has Lebesgue measure zero.

Now suppose $\lambda\in\C\setminus\R$. Write $w=a+ib$ with
$a,b\in\R^N$. The vectors $a$ and $b$ are linearly independent; otherwise
a nonzero real vector would be a left eigenvector associated with the
non-real eigenvalue $\lambda$, which is impossible for a real matrix.
Hence the real-linear map $x\mapsto w^*x=a^\top x-i b^\top x$ maps
$\R^N$ onto $\C$, and therefore $\Phi_w$ maps $\R^{2N}$ onto $\C^2$.
By Lemma~\ref{lem:asymptotics}(iii), the nongrowing scalar data form a
one-dimensional complex subspace $\mathcal S_\lambda\subset\C^2$. Thus
$\Phi_w^{-1}(\mathcal S_\lambda)$ is a proper real linear subspace of
$\R^{2N}$ and has Lebesgue measure zero.

For all initial conditions outside this exceptional set, $|y(t)|$ grows
exponentially. Since
\[
|y(t)|=|w^*x(t)|\leq \|w\|_2\|x(t)\|_2,
\]
we obtain $\|x(t)\|_2\geq |y(t)|/\|w\|_2\to\infty$ exponentially.
\end{proof}

\subsection{Null Space Dynamics}
\begin{lemma}[Null space projection]
\label{lem:nullspace}
If $w \in \N(G^\top) \setminus \{0\}$, then $y_1(t) = w^\top x(t)$
satisfies
\[
y_1(t) =
y_1(t_0)+\frac{t_0}{2}\dot y_1(t_0)
\left(1-\frac{t_0^2}{t^2}\right),
\qquad
\dot{y}_1(t) = \left(\frac{t_0}{t}\right)^3 \dot{y}_1(t_0).
\]
In particular,
\[
y_1(t) \to y_1(t_0)+\frac{t_0}{2}\dot{y}_1(t_0),
\qquad
\dot{y}_1(t)\to 0.
\]
\end{lemma}

\begin{proof}
Projecting~\eqref{eq:NAGD_game} via $w^\top$ and using $w^\top G=0$
gives
\[
\ddot y_1+\frac{3}{t}\dot y_1=0.
\]
Let $p(t)=\dot y_1(t)$. Then $(t^3p(t))'=0$, so
$p(t)=p(t_0)(t_0/t)^3$. Integrating from $t_0$ to $t$ gives the stated
formula.
\end{proof}

\section{Main Results}
\label{sec:main}

\begin{definition}[Bounded stability and instability]
Fix \(t_0>0\). The equilibrium \(x=0\) of \((2)\) is called
bounded-stable if there exists \(\delta>0\) such that, whenever
\[
\|x(t_0)\|+\|\dot x(t_0)\|<\delta,
\]
the corresponding solution satisfies
\[
\sup_{t\ge t_0}\bigl(\|x(t)\|+\|\dot x(t)\|\bigr)<\infty .
\]
It is called asymptotically stable if, in addition,
\[
(x(t),\dot x(t))\to (0,0).
\]
It is called unstable if bounded stability fails.
\end{definition}
This bounded-stability notion is weaker than Lyapunov stability, which
requires trajectories starting near the equilibrium to remain uniformly
near it. The distinction matters when \(G\) has a nontrivial null space:
trajectories may remain bounded and converge to \(\mathcal N(G)\) without
converging to the origin.

\subsection{Symmetric Pseudo-Gradient (Potential Games)}
\begin{theorem}[Stability for potential games]
\label{thm:symmetric}
Let $G = G^\top$ with eigenvalues $\lambda_1, \ldots, \lambda_N \in \R$.
\begin{enumerate}
\item[(a)] If $G \succeq 0$, then $\dot{x}(t) \to 0$ and $x(t) \to x^\infty \in \N(G)$ with rate $O(t^{-3/2})$ for components in $\N(G)^\perp$.
\item[(b)] If $\lambda_j < 0$ for some $j$, there exist arbitrarily small initial conditions with $\|x(t)\|_2 \to \infty$ exponentially.
\end{enumerate}
\end{theorem}
\begin{proof}
Since $G=G^\top$, there is an orthonormal eigenbasis $\{w_i\}_{i=1}^N$.
Writing $y_i(t)=w_i^\top x(t)$, each mode satisfies
\[
\ddot y_i+\frac{3}{t}\dot y_i+\lambda_i y_i=0,
\]
and
\[
x(t)=\sum_{i=1}^N y_i(t)w_i,
\qquad
\dot x(t)=\sum_{i=1}^N \dot y_i(t)w_i.
\]

For part~(a), if $\lambda_i>0$, Lemma~\ref{lem:asymptotics}(i) gives
$y_i(t)=O(t^{-3/2})$ and $\dot y_i(t)=O(t^{-3/2})$. If $\lambda_i=0$,
Lemma~\ref{lem:asymptotics}(iv) gives convergence of $y_i(t)$ and
$\dot y_i(t)\to0$. Hence $\dot x(t)\to0$, while $x(t)$ converges to the
sum of the limiting zero-eigenvalue components, which lies in $\N(G)$.
The components in $\N(G)^\perp$ correspond exactly to the positive
eigenvalues and therefore decay at rate $O(t^{-3/2})$.

For part~(b), choose an eigenvector $w_j$ associated with
$\lambda_j<0$ and take initial data supported only in this eigendirection.
By Lemma~\ref{lem:asymptotics}(ii), all such scalar initial data outside
a one-dimensional exceptional subspace grow exponentially. These initial
conditions can be scaled arbitrarily small, proving instability.
\end{proof}

\begin{corollary}\label{cor:potential}
For potential games with $G = G^\top \succ 0$: $(x(t), \dot{x}(t)) \to (0, 0)$.
\end{corollary}

\subsection{Normal Pseudo-Gradient}
\begin{theorem}[Stability for normal $G$]
\label{thm:normal}
Let $G$ be normal ($GG^\top = G^\top G$) with eigenvalues
$\{\lambda_i\} \subset \C$.
\begin{enumerate}
\item[(a)] If $\lambda_i \in \R_{\geq 0}$ for all $i$, the conclusions of
Theorem~\ref{thm:symmetric}(a) hold. (Note: a real normal matrix with all
eigenvalues in $\R_{\geq 0}$ is necessarily symmetric, so this case
reduces to the potential game setting.)
\item[(b)] If $\lambda_j \notin \R_{\geq 0}$ for some $j$, then
Lebesgue-almost every initial condition leads to exponential growth. In
particular, there exist initial conditions leading to exponential growth.
\end{enumerate}
\end{theorem}

\begin{proof}
For part~(a), since $G$ is real and normal with
$\lambda_i\in\R_{\geq0}$ for all $i$, it has real spectrum. A real normal
matrix with real spectrum is symmetric, so the claim follows from
Theorem~\ref{thm:symmetric}(a).

Part~(b) follows directly from Corollary~\ref{cor:generic}.
\end{proof}

\begin{corollary}[Complex eigenvalues with positive real parts]
\label{cor:complex}
Let $G$ be normal, suppose $\Re(\lambda_i)>0$ for all eigenvalues
$\lambda_i$, and suppose at least one eigenvalue has the form
$\lambda_j=a+ib$ with $b\neq0$. Then the first-order dynamics
$\dot{x}=-Gx$ are exponentially stable with rate
$\min_i \Re(\lambda_i)$, while NAGD~\eqref{eq:NAGD_game} is
exponentially unstable. Generic modal data associated with $\lambda_j$
grow at rate
\[
|\im(\sqrt{\lambda_j})|
=
\sqrt{\frac{\sqrt{a^2+b^2}-a}{2}}.
\]
\end{corollary}

\begin{corollary}[Purely imaginary modes]
\label{cor:imaginary}
If $G$ has a purely imaginary eigenvalue $\lambda=ib$, $b\neq0$, then
the corresponding first-order mode is marginal, while generic
corresponding NAGD modal data grow exponentially with growth rate
$\sqrt{|b|/2}$. In particular, if $G$ is skew-symmetric and has a nonzero
eigenvalue $ib$, then first-order dynamics are marginally stable, while
NAGD~\eqref{eq:NAGD_game} is exponentially unstable.
\end{corollary}

\subsection{General (Non-Normal) Pseudo-Gradient}
\begin{theorem}[General characterization]
\label{thm:general}
Let $G \in \R^{N \times N}$ be arbitrary.
\begin{enumerate}
\item[(a)] If $G$ has an eigenvalue $\lambda \notin \R_{\geq 0}$, then
Lebesgue-almost every initial condition leads to exponential growth. In
particular, there exist initial conditions with exponential growth.
\item[(b)] If all eigenvalues are in $\R_{\geq 0}$ and $G$ is
diagonalizable, all trajectories are bounded and converge:
$x(t) \to x^\infty \in \N(G)$, $\dot{x}(t) \to 0$, with
\begin{multline}
\sup_{t \geq t_0}\|x(t)\|_2 \leq \\
\kappa(P)\bigl(\|x(t_0)\|_2
  + C(t_0, \{\lambda_i\})\|\dot{x}(t_0)\|_2\bigr),
\end{multline}
where $P$ diagonalizes $G$ and
\[
C(t_0,\{\lambda_i\})
=
\max\{t_0/2,\lambda_{\min}^{-1/2}\},
\]
with $\lambda_{\min}=\min\{\lambda_i:\lambda_i>0\}$; if all eigenvalues
are zero, set $C=t_0/2$.
\end{enumerate}
\end{theorem}

\begin{proof}
Part~(a) follows immediately from Corollary~\ref{cor:generic}.

For part~(b), let $G=P\Lambda P^{-1}$, where
$\Lambda=\operatorname{diag}(\lambda_1,\ldots,\lambda_N)$ and
$\lambda_i\in\R_{\geq0}$. Define
\[
y(t)=P^{-1}x(t).
\]
Then $x(t)=P y(t)$, $\dot x(t)=P\dot y(t)$, and each coordinate satisfies
\[
\ddot y_i+\frac{3}{t}\dot y_i+\lambda_i y_i=0.
\]

If $\lambda_i>0$, Lemma~\ref{lem:asymptotics}(i) gives
$y_i(t)\to0$ and $\dot y_i(t)\to0$. If $\lambda_i=0$,
Lemma~\ref{lem:asymptotics}(iv) gives convergence of $y_i(t)$ and
$\dot y_i(t)\to0$. Therefore $y(t)$ converges and $\dot y(t)\to0$.
Since the positive-eigenvalue coordinates converge to zero, the limit
$x^\infty=P y^\infty$ lies in the span of the zero-eigenvalue right
eigenvectors, hence $x^\infty\in\N(G)$. Also,
$\dot x(t)=P\dot y(t)\to0$.

It remains to prove the stated uniform bound. For a mode with
$\lambda_i>0$, define the modal energy
\[
E_i(t)=\frac{1}{2}|\dot y_i(t)|^2+\frac{\lambda_i}{2}|y_i(t)|^2.
\]
Taking real parts after multiplying the scalar equation by
$\overline{\dot y_i(t)}$ gives
\[
\dot E_i(t)
=
-\frac{3}{t}|\dot y_i(t)|^2
\leq 0.
\]
Hence $E_i(t)\leq E_i(t_0)$, and therefore
\[
|y_i(t)|
\leq
\left(|y_i(t_0)|^2+\lambda_i^{-1}|\dot y_i(t_0)|^2\right)^{1/2}
\leq
|y_i(t_0)|+\lambda_i^{-1/2}|\dot y_i(t_0)|.
\]
For a mode with $\lambda_i=0$, Lemma~\ref{lem:asymptotics}(iv) gives
\[
|y_i(t)|
\leq
|y_i(t_0)|+\frac{t_0}{2}|\dot y_i(t_0)|.
\]
Thus, for every mode,
\[
|y_i(t)|\leq |y_i(t_0)|+
C(t_0,\{\lambda_i\})|\dot y_i(t_0)|.
\]
Taking the Euclidean norm and using the triangle inequality,
\[
\|y(t)\|_2
\leq
\|y(t_0)\|_2+
C(t_0,\{\lambda_i\})\|\dot y(t_0)\|_2.
\]
Finally,
\[
\|y(t_0)\|_2\leq \|P^{-1}\|_2\|x(t_0)\|_2,
\qquad
\|\dot y(t_0)\|_2\leq \|P^{-1}\|_2\|\dot x(t_0)\|_2,
\]
and
\[
\|x(t)\|_2\leq \|P\|_2\|y(t)\|_2.
\]
Combining these inequalities gives the stated bound with
$\kappa(P)=\|P\|_2\|P^{-1}\|_2$.
\end{proof}
\begin{remark}[Non-diagonalizable case]
\label{rem:jordan}
When $G$ has non-trivial Jordan blocks with eigenvalues in $\R_{\geq 0}$, generalized eigenvector contributions may cause polynomial growth. The instability direction (part~(a)) is unconditional; only the convergence direction (part~(b)) requires diagonalizability. A complete characterization of the Jordan block case is left to future work.
\end{remark}

\begin{remark}[Bound degeneracy for small eigenvalues]
\label{rem:bound}
The constant $C(t_0,\{\lambda_i\})$ in Theorem~\ref{thm:general}(b)
grows as $\lambda_{\min}^{-1/2}$ when the smallest positive eigenvalue
approaches zero. This reflects the loss of coercivity in the modal energy
estimate for small positive eigenvalues, not a failure of convergence:
each mode still converges by Lemma~\ref{lem:asymptotics}(i,iv), but the
transient bound on $\|x(t)\|_2$ degrades for near-semidefinite~$G$.
\end{remark}
\subsection{Summary}
Table~\ref{tab:comparison} summarizes the stability conditions.

\begin{table}[!ht]
\renewcommand{\arraystretch}{1.3}
\caption{Stability Conditions for Equilibrium Seeking}
\label{tab:comparison}
\centering
\begin{tabular}{|l|l|}
\hline
\textbf{Dynamics} & \textbf{Condition / conclusion} \\
\hline
$\dot{x} = -Gx$ & $\Re(\lambda_i) > 0$ for all $i$ \\
\hline
$\ddot{x} + \frac{3}{t}\dot{x} + Gx = 0$
& unstable if some $\lambda_i\notin\R_{\geq0}$ \\
\hline
$\ddot{x} + \frac{3}{t}\dot{x} + Gx = 0$
& convergent if $\sigma(G)\subset\R_{\geq0}$ and $G$ diagonalizable \\
\hline
\end{tabular}
\end{table}

\section{Numerical Simulations}
\label{sec:simulations}

All simulations use fourth-order Runge--Kutta with $\Delta t = 0.01$ and $t_0 = 1$.

\subsection{Stable: Symmetric Positive Definite $G$}
Figure~\ref{fig:symmetric_pd} verifies Theorem~\ref{thm:symmetric}(a) with $G = \left[\begin{smallmatrix} 0.4 & 0.2 \\ 0.2 & 0.8 \end{smallmatrix}\right]$ (eigenvalues $\approx 0.32, 0.88$). The trajectory converges at the predicted $O(t^{-3/2})$ rate with oscillations reflecting the Bessel function structure.

\subsection{Unstable: Complex Eigenvalues with Positive Real Parts}
Figure~\ref{fig:complex} demonstrates the central finding using $G = \left[\begin{smallmatrix} 6 & 1.5 \\ -1.5 & 6 \end{smallmatrix}\right]$ (eigenvalues $6 \pm 1.5i$). First-order dynamics decay at rate $e^{-6t}$, while NAGD grows exponentially at rate $|\im(\sqrt{\lambda})| \approx 0.30$, confirming Corollary~\ref{cor:complex}. This setting also subsumes the zero-sum case of Corollary~\ref{cor:imaginary} as the special case $a = 0$.

\subsection{Unstable: Negative Real Eigenvalue}
Figure~\ref{fig:negative} illustrates Theorem~\ref{thm:symmetric}(b) with $G = \mathrm{diag}(1, -0.5)$. The stable mode ($\lambda_1 = 1$) decays while the unstable mode ($\lambda_2 = -0.5$) grows at the predicted rate $e^{\sqrt{0.5}\,t}$.

\subsection{Semidefinite: Convergence to Null Space}
Figure~\ref{fig:semidefinite} demonstrates Theorem~\ref{thm:symmetric}(a) for $G = \frac{1}{4}\left[\begin{smallmatrix} 1 & 1 \\ 1 & 1 \end{smallmatrix}\right]$ (eigenvalues $0, 0.5$). The trajectory converges to $\N(G) = \mathrm{span}\{(1,-1)^\top\}$ at rate $O(t^{-3/2})$.

\subsection{Multiplayer Stable Games}
Figure~\ref{fig:multiplayer} extends the results to $3$- and $4$-player potential games with symmetric positive definite $G$.
For the 3-player game, we use the symmetric positive definite matrix
\begin{equation}
G_3 = \begin{bmatrix} 1.0 & 0.3 & 0.2 \\ 0.3 & 0.8 & 0.25 \\ 0.2 & 0.25 & 0.6 \end{bmatrix},
\end{equation}
with eigenvalues $\lambda \approx 0.43, 0.62, 1.35$. Panel (a) shows the 3D trajectory spiraling to the origin, while panel (b) confirms the $O(t^{-3/2})$ convergence rate. For the 4-player game, we consider
\begin{equation}
G_4 = \begin{bmatrix} 1.2 & 0.2 & 0.15 & 0.1 \\ 0.2 & 0.9 & 0.2 & 0.15 \\ 0.15 & 0.2 & 0.7 & 0.1 \\ 0.1 & 0.15 & 0.1 & 0.5 \end{bmatrix},
\end{equation}
with eigenvalues $\lambda \approx 0.44, 0.58, 0.87, 1.41$. Panel (c) displays all four components converging to zero with damped oscillations characteristic of Bessel function solutions, and panel (d) verifies the theoretical $O(t^{-3/2})$ decay rate.

\begin{figure}[!t]
\centering
\includegraphics[width=\linewidth]{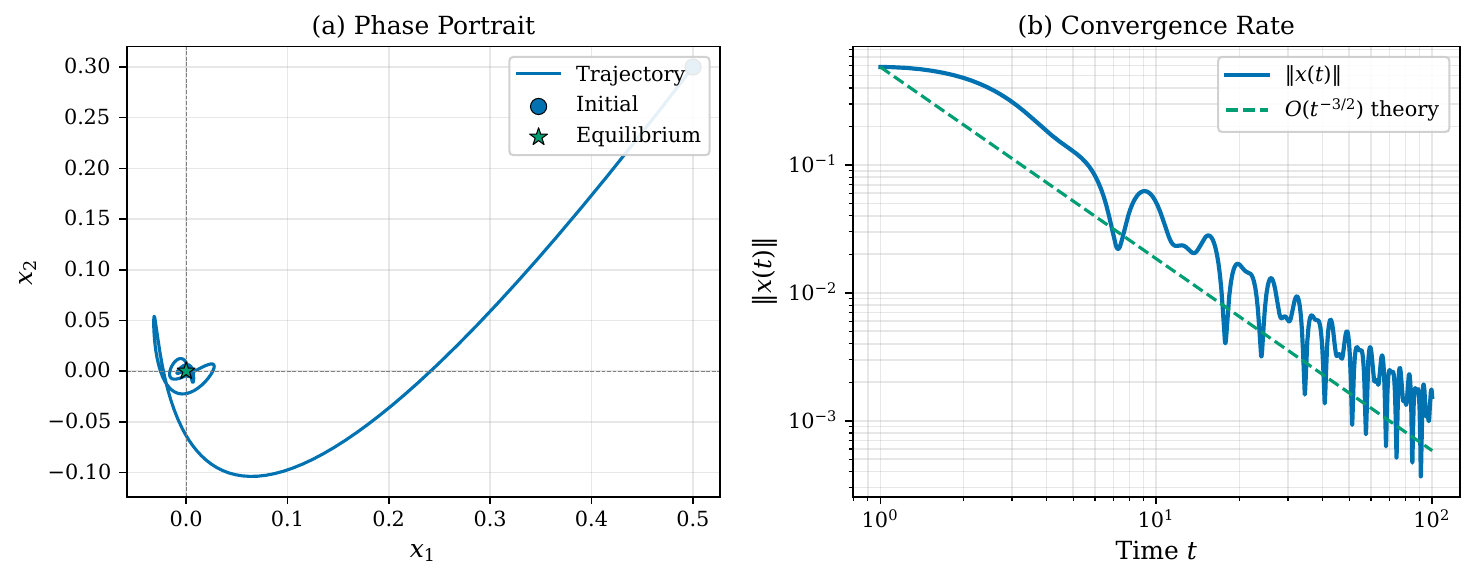}
\caption{Symmetric positive definite $G$ (eigenvalues $0.32, 0.88$): (a)~phase portrait; (b)~$O(t^{-3/2})$ decay confirming Theorem~\ref{thm:symmetric}(a).}
\label{fig:symmetric_pd}
\end{figure}

\begin{figure}[!t]
\centering
\includegraphics[width=\linewidth]{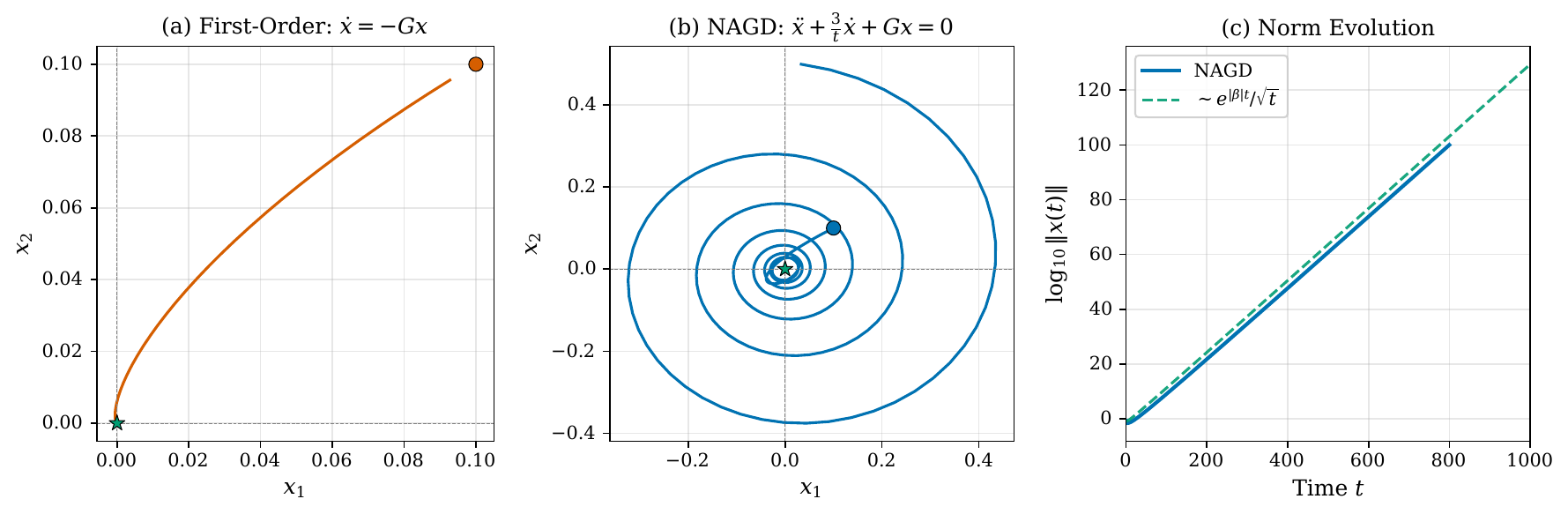}
\caption{Normal $G$ with eigenvalues $6 \pm 1.5i$: (a)~first-order dynamics converge; (b)~NAGD diverges; (c)~exponential separation confirming Corollary~\ref{cor:complex}.}
\label{fig:complex}
\end{figure}

\begin{figure}[!t]
\centering
\includegraphics[width=\linewidth]{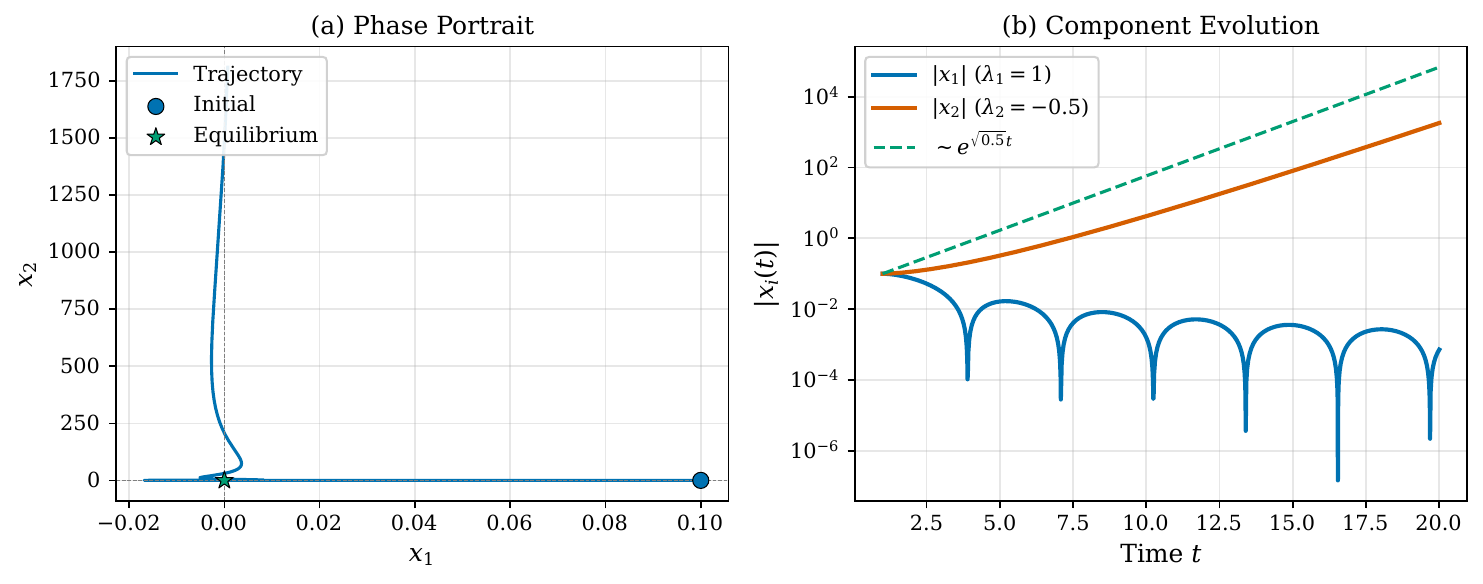}
\caption{Diagonal $G$ (eigenvalues $1, -0.5$): (a)~unbounded growth in unstable direction; (b)~growth rate $e^{\sqrt{0.5}\,t}$ confirms Theorem~\ref{thm:symmetric}(b).}
\label{fig:negative}
\end{figure}

\begin{figure}[!t]
\centering
\includegraphics[width=\linewidth]{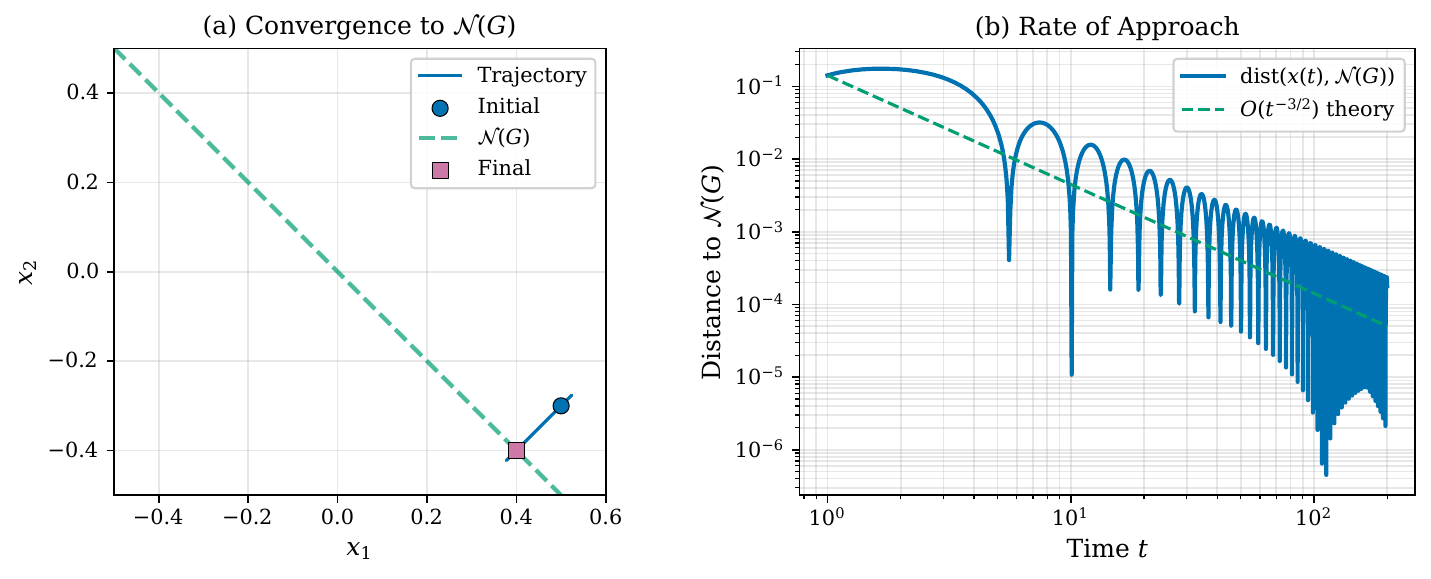}
\caption{Positive semidefinite $G$ (eigenvalues $0, 0.5$): (a)~convergence to $\N(G)$; (b)~$O(t^{-3/2})$ approach rate.}
\label{fig:semidefinite}
\end{figure}

\begin{figure}[!t]
\centering
\includegraphics[width=\linewidth]{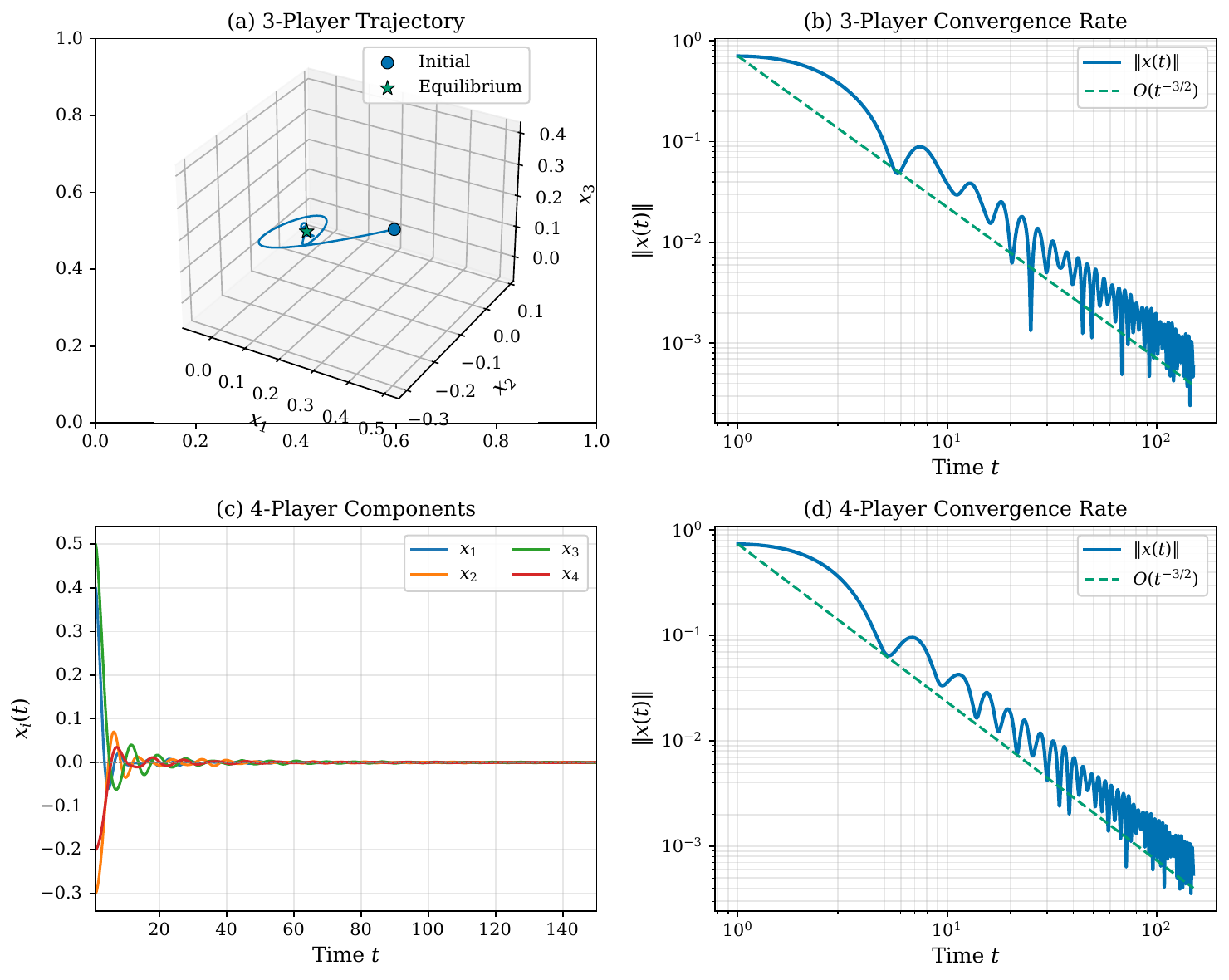}
\caption{Multiplayer potential games: (a,b)~3-player; (c,d)~4-player. All confirm $O(t^{-3/2})$ decay from Theorem~\ref{thm:symmetric}(a).}
\label{fig:multiplayer}
\end{figure}

\section{Conclusion}
\label{sec:discussion}

We have shown that NAGD is unstable for Nash equilibrium seeking whenever any eigenvalue of $G$ lies outside $\R_{\geq 0}$, and that all trajectories converge when eigenvalues lie in $\R_{\geq 0}$ and $G$ is diagonalizable. This is strictly more restrictive than the $\Re(\lambda_i) > 0$ condition for first-order dynamics: complex eigenvalues with positive real parts, ubiquitous in non-potential games, destabilize NAGD. For potential games ($G = G^\top \succeq 0$), NAGD retains its optimization benefits; for general games, including zero-sum games, first-order or extragradient methods \cite{korpelevich1976} may be preferable.

Future directions include the non-diagonalizable case, extension to nonlinear games via linearization, general damping $r \neq 3$, discrete-time schemes, and modified momentum methods stable for broader game classes.

\section*{Use of Generative AI}
Anthropic's Claude \cite{claude} was used to assist in writing Python code for the numerical simulations in Section~V. The authors take full responsibility for the correctness of all technical content, proofs, and final text.

\appendices
\section{Proof of Lemma~\ref{lem:asymptotics}(iii)}
\label{app:proof_iii}

\begin{proof}
Let $\lambda\in\C\setminus\R$ and write
\[
\sqrt{\lambda}=\alpha+i\beta,
\qquad
\alpha\geq0,
\qquad
\beta\neq0.
\]
We first consider $\beta>0$; the case $\beta<0$ is identical with the
roles of $H_1^{(1)}$ and $H_1^{(2)}$ interchanged.

In the Hankel basis, the general solution is
\[
y(t)
=
\frac{1}{t}
\left[
c_1 H_1^{(1)}(\sqrt{\lambda}\,t)
+
c_2 H_1^{(2)}(\sqrt{\lambda}\,t)
\right].
\]
The large-argument asymptotics
\[
H_1^{(1)}(z)\sim \sqrt{\frac{2}{\pi z}}e^{i(z-3\pi/4)},
\qquad
H_1^{(2)}(z)\sim \sqrt{\frac{2}{\pi z}}e^{-i(z-3\pi/4)}
\]
give, with $z=(\alpha+i\beta)t$,
\[
\frac{1}{t}H_1^{(1)}(\sqrt{\lambda}\,t)
=
O(t^{-3/2}e^{-\beta t}),
\qquad
\frac{1}{t}H_1^{(2)}(\sqrt{\lambda}\,t)
=
O(t^{-3/2}e^{\beta t}).
\]
Moreover, the second estimate has a nonzero leading asymptotic coefficient.
Thus the solution grows exponentially at rate $\beta$ unless $c_2=0$.
The condition $c_2=0$ defines a one-dimensional complex subspace
$\mathcal S_\lambda\subset\C^2$ of scalar initial data.

It remains to show that this exceptional subspace contains no nonzero real
scalar initial data. Suppose, for contradiction, that
\[
(y(t_0),\dot y(t_0))\in\R^2\setminus\{(0,0)\}
\]
and $c_2=0$. Then
\[
y(t)=O(t^{-3/2}e^{-\beta t}),
\qquad
\dot y(t)=O(t^{-3/2}e^{-\beta t}).
\]
The derivative estimate follows by differentiating the decaying Hankel
asymptotic; the leading differentiated exponential factor remains of order
$t^{-3/2}e^{-\beta t}$.

Since the initial data are real, define
\[
Q(t)=y(t)\overline{\dot y(t)}-\dot y(t)\overline{y(t)}.
\]
Using
\[
\ddot y+\frac{3}{t}\dot y+\lambda y=0,
\qquad
\overline{\ddot y}+\frac{3}{t}\overline{\dot y}
+\overline{\lambda}\,\overline y=0,
\]
a direct computation gives
\[
\frac{d}{dt}\left(t^3Q(t)\right)
=
t^3(\lambda-\overline{\lambda})|y(t)|^2
=
2i\,\im(\lambda)\,t^3|y(t)|^2.
\]
At $t=t_0$, the reality of $y(t_0)$ and $\dot y(t_0)$ gives
$Q(t_0)=0$. On the other hand, the assumed decay gives
\[
t^3Q(t)\to0
\qquad\text{as }t\to\infty.
\]
Integrating from $t_0$ to $\infty$ yields
\[
0
=
2i\,\im(\lambda)
\int_{t_0}^{\infty} t^3 |y(t)|^2\,dt.
\]
Since $\im(\lambda)\neq0$ and the solution is not identically zero, the
integral is strictly positive, a contradiction. Therefore
$\mathcal S_\lambda\cap\R^2=\{(0,0)\}$.

For $\beta<0$, the same argument applies after interchanging the growing
and decaying Hankel modes. Hence, for all scalar initial data outside
$\mathcal S_\lambda$, the solution grows exponentially at rate
$|\beta|=|\im(\sqrt{\lambda})|$, and every nonzero real scalar initial
condition lies outside $\mathcal S_\lambda$.
\end{proof}
\bibliographystyle{IEEEtran}
\bibliography{sources}

\end{document}